\documentclass{amsart}

\usepackage[utf8]{inputenc}
\usepackage{amsmath, amssymb, amsthm}
\usepackage{amsthm}
\usepackage{xcolor}
\usepackage{hyperref}
\usepackage{mathtools}
\DeclareMathOperator{\divergence}{div}

\newtheorem{theorem}{Theorem}[section]

\newtheorem{definition}[theorem]{Definition}
\newtheorem{proposition}[theorem]{Proposition}

\newtheorem{remark}[theorem]{Remark}

\numberwithin{equation}{section}
\begin{document}

\title[A double-phase Dirichlet problem involving the 1-Laplacian]{On the solutions of a double-phase Dirichlet problem involving the 1-Laplacian}
\author{Alexandros Matsoukas }
\address{Department of Mathematics\\
National Technical University of Athens\\
Iroon Polytexneiou 9\\
15780 Zografou\\
Greece}
\email{alexmatsoukas@mail.ntua.gr}
\author{Nikos Yannakakis}
\address{Department of Mathematics\\
National Technical University of Athens\\
Iroon Polytexneiou 9\\
15780 Zografou\\
Greece}
\email{nyian@math.ntua.gr}
\subjclass[2000]{Primary  35J60;  35J25; 35J75 Secondary  46E35;
 35J92; 35D30.}
\commby{}

%% Abstract
\begin{abstract}
%% Text of abstract
In this paper we study a double-phase problem involving the 1-Laplacian with non-homogeneous Dirichlet boundary conditions and show the existence and uniqueness of a solution in a suitable weak sense. We also provide a variational characterization of this solution via the corresponding minimization problem.
\end{abstract}

\maketitle

\noindent\textbf{Keywords:} Double-phase problem, 1-Laplacian, Generalized Orlicz space, Weighted Sobolev space

\section{Introduction}
\label{sec1}

This paper is devoted to the study of the double-phase problem involving the 1-Laplacian with non-homogeneous Dirichlet boundary conditions

\begin{equation}\label{limiting}
\left\{
\begin{array}{rcl}
-\divergence \left(\frac{\nabla u}{|\nabla u|} +a(x)|\nabla u|^{q-2}\nabla u \right)&=&0\,\text{ in } \Omega,\\
u&=&h \, \text{ on }\partial\Omega,\,
\end{array}
\right.
\end{equation}
where $\Omega \subset \mathbb{R}^N$ is a bounded Lipschitz domain, $a(\cdot)$ is a bounded function with $a\ge 0$ a.e. in $\Omega$ and the datum $h$ belongs to the fractional Sobolev space $W^{1-\frac{1}{q},q}(\partial \Omega)$.

The above problem, may be seen as the limiting case as $p\to 1$, of $p,q$ double-phase problems, that is equations driven by the differential operator
\begin{equation}\label{pqdp}
\divergence \left(|\nabla u|^{p-2}\nabla u +a(x)|\nabla u|^{q-2}\nabla u \right)\,\text{ for } \, u \in W^{1,\theta_p}(\Omega),\\
\end{equation}
which is related to the so-called double-phase functional
\begin{equation*}
 u \mapsto \int_{\Omega} (|\nabla u|^{p} +a(x)|\nabla u|^{q})\, d x, 
\end{equation*}
with $1<p<q$.

The double-phase functional was first introduced by Zhikov \cite{Zhikov} to model strongly anisotropic materials with two hardening exponents. It enjoys the interesting feature of obeying non-standard growth conditions of $p,q$ type, according to the terminology of Marcelini \cite{Marcelini1}.
Consequently, the term ``double-phase problems'' is used in the literature to describe this characteristic feature: a change in ellipticity type, which is entirely determined by the function $a(\cdot)$. These different phases are called the p-phase and the q-phase, occurring on the sets $\{a(x)=0\}$ and $\{a(x)>0\}$ respectively. Following the pioneering works of Mingione and co-workers \cite{mingione_1}, \cite{mingione_2}, double-phase problems have attracted significant interest and have been extensively studied by many authors (see for example \cite{Marcelini2}, \cite{MinRad}, \cite{P}, \cite{PRY} \cite{pap_2}, \cite{R} and the references therein). Note that the case $p=1$, which is the topic of this paper, has been tackled in \cite{EHH1}, \cite{EHH2}, \cite{GLM}, \cite{harjulehto_1}, \cite{MY} and hence seems to be quite overlooked. One of the aims of the present study is an attempt to fill this gap.

In \cite{MY} a double-phase problem involving the 1-Laplacian was studied and its solution was found as the limit of solutions of approximate $p,q$ double-phase problems, as $p\to 1$. In this paper we will follow the same approach in order to show that there exists a suitably defined weak solution of problem \eqref{limiting}, which is unique due to the regularizing effect of the weighted term. Additionally we will also provide a variational characterization of this solution via the corresponding minimization problem.   

The natural function space to look for a solution of problem \eqref{limiting} is $ W^{1,1}(\Omega)\cap W^{1,q}_{a}(\Omega)$ where $W^{1,q}_{a}(\Omega)$ denotes a suitable weighted Sobolev space. However, since the space $W^{1,1}(\Omega)$ lacks important compactness properties, we will also use the larger space $BV(\Omega)$ of functions of bounded variation. 

Let us here make some important observations. First, we need to give sense to $\frac{\nabla u}{|\nabla u|}$ which appears in the formal definition of the 1-Laplacian operator, especially when $\nabla u$ vanishes on a non-negligible set. In order to overcome this difficulty, we replace the above quotient by a bounded vector field $z$. This is by now a standard idea for problems involving the 1-Laplacian, see for instance the work of Maz\'{o}n, Rossi and Segura De Le\'{o}n \cite{MRS}, in relation to the least gradient problem and 1-harmonic maps where the authors, motivated by the definition of solution for the total vatiation flow \cite{ABCM}, and using the theory of Anzellotti \cite{Anzellotti}, introduced a notion of solution based on a suitable pairing $(z,Du)$, between a vector field $z$ and the measure $Du$ for $u\in BV(\Omega)$. This pairing serves as a generalization of the inner product and whenever $Du=\nabla u \, \mathcal{L}^N$, as it happens in our case, it reduces to $(z,Du)=z\cdot \nabla u $. For an extensive overview of the least gradient problem, we refer the interested reader to the monograph \cite{gornymazon}.

Another usual difficulty when dealing with the 1-Laplacian is to give a precise meaning to the boundary condition. Note that as the solution $u$ is obtained as a weak* limit in $BV(\Omega)$, we need to address the lack of
 weak* continuity of the trace operator there. A way to overcome this obstacle, as was also done in \cite{MY}, is to assume that the weight function $a$ is bounded away from zero on $\partial \Omega$. Then a trace may be defined in the corresponding weighted Sobolev space $W^{1,q}_{a}(\Omega)$ and the boundary condition has to be satisfied in this sense.

\section{Notation and preliminaries}

In this section, we recall some definitions of the function spaces involved in our analysis.

\subsection{Generalized Orlicz spaces} For this part we follow mainly the survey paper \cite{P}. Let $\Omega \subset \mathbb{R}^N$ be a bounded Lipschitz domain, with $1<p<q<N$ and let $a\in L^{\infty}(\Omega)$  be a non-negative weight function. For fixed $q$ and with $p$ taking values in the above range, the functions
\[\theta_p :\Omega\times\mathbb R_+\rightarrow\mathbb R_+\]
defined by
$$ \theta_p(x,t)=t^p+a(x)t^q $$
are uniformly convex, generalized $\Phi$-functions \cite[Remark 2.22]{P} and satisfy the $(\Delta_2)$ condition \cite[Proposition 2.6]{P}.
The double-phase generalized Orlicz space is defined as
$$L^{\theta_p}(\Omega)=\{u: \Omega \to \mathbb{R}\, \text{measurable} : \rho_{\theta_p} (u) < +\infty \},$$
with modular given by
$$\rho_{\theta_p}(u)=\int_{\Omega} \theta_p(x,|u(x)|)\,d x.$$
When equipped with the so-called Luxemburg norm
$$\|u\|_{\theta_p} = \inf \{ \lambda>0 : \rho_{\theta_p} (\frac{u}{\lambda}) \le 1\},$$
$L^{\theta_p}(\Omega)$ becomes a uniformly convex (and hence reflexive) Banach space \cite[Proposition 2.23]{P}. 

\noindent The generalized Orlicz-Sobolev space is defined as
$$W^{1,\theta_p}(\Omega)=\left\{u\in L^{\theta_p}(\Omega): |\nabla u|\in L^{\theta_p}(\Omega)\right\},$$
where $\nabla u$ is the weak gradient of $u$, and equipped with the norm
$$\|u\|_{W^{1,\theta_p}}=\|u\|_{\theta_p} + \|\nabla u\|_{\theta_p}\,.$$
is a reflexive Banach space.

\noindent As usual we define
\[W^{1,\theta_p}_0(\Omega)=\overline{C_{0}^{\infty}(\Omega)}^{\|\cdot\|_{1,\theta_p}}\,.\]

\noindent If $a\in C^{0,1}(\Omega)$ and $\frac{q}{p}<1+\frac{1}{N}$,
then the maximal operator is bounded on $L^{\theta_p}(\Omega)$ and the constant that bounds it is independent of $p$ (see \cite[Theorem 4.3.4]{harjbook}). Hence the Poincare inequality 
\begin{equation}
    \label{poincare}
\|u\|_{\theta_p} \leq C \|\nabla u\|_{\theta_p}
\end{equation}
holds for all $u\in W_0^{1,\theta_p}(\Omega)$ and the constant $C$ may also be chosen independently of $p$ (see \cite[Theorem 6.2.8]{harjbook}).  

For more details on generalized Orlicz spaces we refer the interested reader to the book \cite{harjbook}.

\subsection{Weighted Lebesgue and Sobolev spaces}
We begin with the definition of the Muckenhoupt class $A_q$.
\begin{definition}
A weight $a\in L^\infty(\Omega)$ with $a(x)>0$ a.e. in $\Omega$ belongs to the Muckenhoupt class $A_q$ if
\[\sup_Q\left(\frac{1}{|Q|}\int_Q a(x) \, d x\right) \left(\frac{1}{|Q|}\int_Q a(x)^{-\frac{1}{q-1}} \, d x\right)^{q-1} < \infty,\]
where the supremum is taken over all cubes $Q$ with sides parallel to the coordinate axes.
\end{definition}
\noindent From now on we will always assume that the weight $a$ belongs to $A_q$. The weighted Lebesgue and Sobolev spaces are defined as
$$L^q_a (\Omega)=\{u: \Omega \to \mathbb{R}\, \text{measurable} \, : \int_{\Omega} a(x)|u|^q \, d x <+\infty\}$$
and
$$W^{1,q}_{a}(\Omega)=\left\{u\in L^{q}_{a}(\Omega): |\nabla u|\in L^{q}_{a}(\Omega)\right\}\,.$$
Equipped with the norms 
$$\|u\|_{L^q_a} = \left( \int_{\Omega} a(x)|u|^q \, d x \right)^{\frac{1}{q}} \text{ and } \|u\|_{W^{1,q}_a}=\|u\|_{L^q_a}+\|\nabla u\|_{L^q_a}$$ $L^q_a (\Omega)$ and $W^{1,q}_{a}(\Omega)$  become reflexive Banach spaces.

\noindent As before 
\[W^{1,q}_{a, 0}(\Omega)=\overline{C_{0}^{\infty}(\Omega)}^{\|\cdot\|_{{W^{1,q}_a}}}\,.\]

%%%%%%%%%%%%%%%%
If $\Omega$ is a bounded Lipschitz domain, $a\in C(\overline{\Omega})$ and is non-zero on $\partial\Omega$ then we can define a trace on $W^{1,q}_{a}(\Omega)$. 
%%%%%%%%%%%%%%%%%%%%%%%%%%%%%%%%%%%%%%%%%%%%%%%%%%%%%%%%%%%%%%  
\begin{proposition}[\cite{MY}, Proposition 2.2]
\label{prop:trace}
Let $a\in C(\overline{\Omega})$ with $a\geq 0$ a.e. in $\Omega$, such that $a(x)\neq 0$, for all $x\in \partial\Omega$. Then there exists a bounded linear operator
\[T:W^{1,q}_{a}(\Omega)\rightarrow L^q(\partial\Omega)\]
such that
\[Tu=u|_{\partial \Omega} \text{ for all } u\in C(\overline{\Omega})\cap W^{1,q}_{a}(\Omega).\]
\end{proposition}

\begin{remark}
\label{gagliardo}
   By a well-known theorem of Gagliardo (see \cite{Gagliardo}) we have that for a Lipschitz domain $\Omega$, the range of the trace operator on $W^{1,q}(\Omega)$ is $W^{1-\frac{1}{q},q}(\partial \Omega)$. Moreover, there exists a bounded linear operator 
   $$\mathcal{E}:  W^{1-\frac{1}{q},q}(\partial \Omega )\to W^{1,q}(\Omega)$$ through which any function $\psi \in W^{1-\frac{1}{q},q}(\partial \Omega )$ can be extended to a function $v\in W^{1,q}(\Omega )$ such that $v|_{\partial \Omega}=\psi$, in the trace sense. Note that under the assumptions of Proposition \ref{prop:trace}, the trace of a function $v\in W^{1,q}_{a}(\Omega)$ belongs to $W^{1-\frac{1}{q},q}(\partial \Omega )$.
\end{remark}

The following Meyers-Serrin type approximation theorem will play a significant role in what follows. 
  
\begin{proposition}\label{prop:MS}
Assume that $a\in A_{q}$. If $u\in W^{1,1}(\Omega)\cap W^{1,q}_{a}(\Omega)$, then there exists a sequence $(v_n)$ in $W^{1,1}(\Omega)\cap C^\infty(\Omega)$ such that
\begin{align*}
v_n &\to u\,, \text{ in } W^{1,1}(\Omega),\\
\nabla v_n &\to \nabla u\,, \text{ in } L^{q}_{a}(\Omega;\mathbb{R}^N).
\end{align*}
If in addition $a(\cdot)$ is as in Proposition \ref{prop:trace} then $v_{n}|_{\partial \Omega} = u|_{\partial \Omega}$, in the sense of the trace in $W^{1,q}_{a}(\Omega)$.
\end{proposition}
\begin{proof}
If $a\in A_q$ then by Muckenhoupt's theorem the maximal operator is bounded in $L^{q}_{a}(\Omega)$ and hence the smoothing operators are uniformly bounded in this space. By this we get that $C^\infty (\Omega)$ is dense in $W^{1,q}_{a}(\Omega)$ (see \cite[Section 4]{Zhikov2}). The proof then proceeds more or less as in the classical case. The equality on $\partial \Omega$ follows as in \cite[Remark 3.5]{MY}.    
\end{proof}

\begin{remark}  
\label{remark}
Under the assumption that $a(x)\neq 0$ for all $x\in \partial\Omega$, we can actually take $(v_n)$ in $W^{1,q}(\Omega)\cap C^\infty(\Omega)$.
\end{remark}

\subsection{Functions of bounded variation}
In this part we follow the books \cite{Buttazzo} and \cite{gornymazon}. A function $u \in L^1(\Omega)$ belongs to $BV(\Omega)$ if its distributional derivative $Du$ is a finite Radon measure.
The total variation of the measure $Du$ is given by
\[|Du|(\Omega)=\sup \{ \langle Du, \phi \rangle : \,\phi \in C^\infty_0(\Omega),\, \|\phi \|_\infty\le 1\}.\]
When equipped with the norm
\[\|u\|_{BV}=\|u\|_1+|Du|(\Omega)\,,\]
the space $BV(\Omega)$ becomes a Banach space and possesses the following important compactness property: if $(u_n)$ is a bounded sequence in $BV(\Omega)$ then there exists a subsequence $(u_{n_k})$ and a function $u\in BV(\Omega)$ such that
\begin{equation*}
u_{n_k}\rightarrow u \, \text{ in } \, L^1(\Omega) \text{ and } \, Du_{n_k}\rightarrow Du\,  \text{ weak* as measures in }\, \Omega.
\end{equation*}
The Lebesgue decomposition of the measure $Du$ is
\begin{equation*}
    Du=\nabla u \, \mathcal{L}^N + D^su,
\end{equation*}
where $\nabla u$ and $D^s u$ denote its absolutely continuous and singular parts with respect to the Lebesgue measure $\mathcal{L}^N$. This decomposition shows that $W^{1,1}(\Omega)$ is a subspace of  $BV(\Omega )$ and $u \in  W^{1,1}(\Omega )$ iff $Du =\nabla u \, \mathcal{L}^N$. For functions in $W^{1,1}(\Omega )$ we will write  $\nabla u$ instead of $Du$.

 \begin{remark}\label{Derivative}
If $u\in BV(\Omega)\cap W_{a}^{1.q}(\Omega)$, then its distributional derivative is a function $g\in L^1_{loc}(\Omega; \mathbb R^N)$ and at the same time a finite Radon measure $Du$. Hence there exists $c>0$ such that 
$$\sup \{ \int_{\Omega}g\,\phi \,d x : \,\phi \in C^\infty_0(\Omega)^N,\, \|\phi \|_\infty\le 1\}\le c\,,$$
which implies that $g\in L^{1}(\Omega;\mathbb{R}^N)$. Thus $u\in W^{1,1}(\Omega)\cap W_{a}^{1.q}(\Omega)$.
 \end{remark}

\noindent In $BV(\mathbb R^N)$ the following Sobolev inequality holds (see \cite[Theorem A.10]{gornymazon})
$$\|u\|_{L^{\frac{N}{N-1}}(\mathbb R^N)}\leq C |Du|(\mathbb{R}^N)\,,\text{ for all }u\in BV(\mathbb R^N)\,.$$

\noindent If additionally we assume that $\Omega$ is a bounded Lipschitz domain then we have the following continuous embedding (see \cite[Theorem A.12]{gornymazon})
$$BV(\Omega) \hookrightarrow L^{s}(\Omega)\,, \text{ for all } 1\le s \le \frac{N}{N-1}$$
which is compact when $1\le s <\frac{N}{N-1}$. Finally, in this latter case using \cite[Theorem A.20]{gornymazon} it can be shown that the norm
%%%%%%%%%%%%%%%%%%%%%%%%%%%%%%%%%%%%%%%%%%%%%
\begin{equation} 
\nonumber
\|u\|=\int_{\Omega}|Du|+\int_{\partial \Omega}|u|\, d \mathcal{H}^{N-1},
\end{equation}
%%%%%%%%%%%%%%%%%%%%%%%%%%%%%%%%%%%%%%%%%%%%%%%%
is equivalent to the usual one of $BV(\Omega)$.
%%%%%%%%%%%%%%%%%%%%%%%%%%%%%%%%%%%%%%%%%%%%%%%%
\section{Main Results}
Our assumptions on the weight function $\alpha(\cdot)$ and the exponents $1<p<q$ are the following.
\begin{equation*}
(H): a\in C^{0,1}(\overline{\Omega})\cap A_{q}, \, a(x) \neq 0 \text{ on } \partial \Omega, \, \text{ and } \, \frac{q}{p}<1+\frac{1}{N}.
\end{equation*}

Our first goal is to prove the existence of a unique weak solution to a suitable approximate double-phase problem. For $h\in W^{1-\frac{1}{q},q}(\partial \Omega)$ let 
%%%%%%%%%%%%%%%%%%%%%%%%%%%%%%%%%%%%%%%%%%%%%%%
$$ W_{h}^{1,\theta_p} (\Omega)=\{ u\in W^{1,\theta_p}(\Omega): u|_{\partial \Omega}=h \, \hspace{2mm} \mathcal{H}^{N-1} \text{- a.e. on } \partial \Omega\}.$$
%%%%%%%%%%%%%%%%%%%%%%%%%%%%%%%%%%%%%%%%%%%%%%
As usual we say that $u \in W_{h}^{1,\theta_p} (\Omega)$ is a weak solution of the double-phase Dirichlet problem
    
\begin{equation}\label{pqdp}
\left\{
\begin{array}{rcl}
-\divergence \left(|\nabla u|^{p-2}\nabla u +a(x)|\nabla u|^{q-2}\nabla u \right)&=&0\,\text{ in }\Omega\\
u&=&h\,\text{ on }\partial\Omega\,,
\end{array}
\right.
\end{equation}
if 
\begin{equation}\label{weakformpq}
    \int_{\Omega} |\nabla u|^{p-2} \nabla u\cdot \nabla v \, d x +\int_{\Omega}a(x)|\nabla u|^{q-2} \nabla u\cdot \nabla v \, d x =0,
\end{equation}
for all $v \in W_0^{1,\theta_p} (\Omega)$.
 
\begin{proposition}\label{existencepq}
Let $h\in W^{1-\frac{1}{q},q}(\partial \Omega)$ and assume that (H) holds. Then, there exists a unique weak solution $u \in W_{h}^{1,\theta_p}(\Omega)$ to the double-phase Dirichlet problem \eqref{pqdp}, which is the unique minimizer of the functional 
\begin{equation*}
   \mathcal{F}(u) = \int_{\Omega} (\frac{|\nabla u|^{p}}{p} +a(x)\frac{|\nabla u|^{q}}{q})\, d x  
\end{equation*}
    in the set  $W_{h}^{1,\theta_p} (\Omega)$.
\end{proposition}

\begin{proof} We will use the direct method of the calculus of variations. 
To this end let $(u_n)\in W_{h}^{1,\theta_p} (\Omega)$ be a minimizing sequence i.e. 
$$\lim_{n\to +\infty} \mathcal{F}(u_n) =\inf \mathcal{F}(u)\,.$$
By Gagliardo's extension theorem (see Remark \ref{gagliardo}) we may extend $h$ to a function $v\in W^{1,q}(\Omega) \hookrightarrow W^{1,\theta_p}(\Omega)$ with $v|_{\partial \Omega}=h$, and so we have that $u_n -v \in W_{0}^{1,\theta_p}(\Omega)$, for all $n\in \mathbb{N}$. By Poincar\'{e}'s inequality we get
\begin{align*}
    \|u_n \|_{\theta_p} = \|u_n -v+v\|_{\theta_p} &\leq  \, \|u_n -v\|_{\theta_p} + \|v \|_{\theta_p} \\
    &\leq C\|\nabla (u_n -v)\|_{\theta_p} + \|v \|_{\theta_p} \\
    &\leq C\|\nabla u_n\|_{\theta_p} + C'\|\nabla v \|_{q} + \|v \|_{q}.
\end{align*}
Since the boundedness of the sequence $\mathcal{F}(u_n)$ implies that $\rho_{\theta_p}(|\nabla u_n|)$ is bounded, we get that $\|\nabla u_n \|_{\theta_p}$ is bounded as well, see \cite[Proposition 2.15c)]{P}. Hence there exists $u\in W^{1,\theta_p}(\Omega)$ such that, up to a subsequence $u_n \stackrel{w}{\rightarrow} u$ in $W^{1,\theta_p}(\Omega)$ and by the weak to weak continuity of the trace $u\in W_{h}^{1,\theta_p} (\Omega)$. By weak lower semicontinuity \cite[Theorem 2.2.8]{DHHR} we get that 
%%%%%%%%%%%%%%%%%%%%%%%%%%%%%%
\begin{equation*}
\int_{\Omega} (\frac{|\nabla u|^{p}}{p} +a(x)\frac{|\nabla u|^{q}}{q})\, d x  \le \liminf_{n}{ \int_{\Omega} (\frac{|\nabla u_n|^{p}}{p}+a(x)\frac{|\nabla u_n|^{q}}{q})\,d x  }
\end{equation*}
%%%%%%%%%%%%%%%%%%%%%%%%%%%%%%%
and hence we conclude that 
%%%%%%%%%%%%%%%%%%%%%%%%%%%%%%%%%
\begin{equation*}
\mathcal{F}(u) \le \liminf_{n}{\mathcal{F}}(u_n) = \lim_{n} \mathcal{F}(u_n)=\inf \mathcal{F}(u). 
\end{equation*}
%%%%%%%%%%%%%%%%%%%%%%%%%%%%%%%%%%%%%
Thus, the infimum of $\mathcal{F}$ is attained and is unique due to the strict convexity of the functional. The fact that this minimizer is a weak solution of problem \eqref{pqdp} is obvious.
\end{proof}

We now give a suitable notion of weak solution for problem \eqref{limiting}. 

\begin{definition}
\label{Notion} 
A function $u\in W^{1,1}(\Omega)\cap W_{a}^{1,q}(\Omega)$ with $u|_{\partial \Omega}=h$, is said to be a weak solution of the Dirichlet problem \eqref{limiting} if there exists a vector field  $z\in L^{\infty}(\Omega)^N$ with $\|z\|_\infty \le 1$, such that 
\begin{eqnarray*}
\int_{\Omega} z \cdot \nabla \phi \, d x+ \int_{\Omega} a(x) |\nabla u|^{q-2} \nabla u \cdot \nabla \phi \, d x &=& 0\,,\text{ for all }\phi\in C^\infty_0(\Omega),\\
z \cdot \nabla u &=& |\nabla u| \, \text{ a.e. in } \Omega.
\end{eqnarray*}
\end{definition}

To proceed to our main result we first study the behavior of
the solutions $(u_p)$ of the approximate problems \eqref{pqdp} as $p \to 1$. To simplify things, with a slight abuse of notation, we will say that $(u_p)$ is a sequence and consider subsequences of it as $p \to 1$.
%%%%%%%%%%%%%%%%%%%%%%%%%%%%%%%%%%%%%%%%%%%%%%%%%%%
\begin{proposition}
\label{prop:behaviour} 
Let $h\in W^{1-\frac{1}{q},q}(\partial \Omega)$ and assume that (H) holds. If $(u_p)$ are the unique weak solutions of problems \eqref{pqdp}, then there exist a function $u\in W^{1,1}(\Omega)\cap W^{1,q}_{a}(\Omega)$ with $u|_{\partial \Omega}=h$ and a vector field $z \in L^{\infty}(\Omega)^N$, with $\|z\|_{\infty}\le 1$ such that as $p\rightarrow 1$, up to subsequences 
\begin{align*}
    u_p &\to u \, \text{ in } L^{s}(\Omega), \text{ for all } 1\le s<\frac{N}{N-1}\,,\\
    |\nabla u_p |^{p-2} \nabla u_p &\stackrel{w}{\rightarrow} z \, \text{ in } L^{r}(\Omega)^N \,, \text{ for all } \, 1\le r < +\infty,\\
 |\nabla u_p|^{q-2} \nabla u_p &\stackrel{w}{\rightarrow}  |\nabla u|^{q-2} \nabla u \, \text{ in } \, L^{q'}_{a}(\Omega)^N\,,\\
 \nabla u_p &\to \nabla u \, \text{ in } \, L^{q}_{a}(\Omega)^N\,.\\
\end{align*}

\end{proposition}
%%%%%%%%%%%%%%%%%%%%%%%%%%%%%%%%%%%%%%%%%%%%%%%%%%%%%%%%%%%%%%%%
\begin{proof}
First we will show that $(\|\nabla u_p\|_{\theta_p})$, is bounded for $p$ near 1. Note that without loss of generality we may assume that $\|\nabla u_p\|_{\theta_p} >1$, for all $p$.

By Gagliardo's extension theorem there exists $v\in W^{1,q}(\Omega) \hookrightarrow W^{1,\theta_p}(\Omega)$ with $v|_{\partial \Omega}=h$. Taking $v$ in the weak formulation \eqref{weakformpq} and using H\"{o}lder's inequality, the embeddings $L^{\theta_p}(\Omega)\hookrightarrow L^{p}(\Omega)$, $L^{\theta_p}(\Omega)\hookrightarrow L_{a}^{q}(\Omega)$ and the fact that $\|\nabla u_p\|_{\theta_p} >1$ we get
%%%%%%%%%%%%%%%%%%%%%%%%%%%%%%%%%%%%%%%%%%%%%%%%%
\begin{eqnarray}
\nonumber \text{I}_p&=&\int_{\Omega} \left(|\nabla u_p|^{p} +a(x)|\nabla u_p|^{q} \right) \, d x \\
\nonumber &=& \int_{\Omega} \left(|\nabla u_p|^{p-2} \nabla u_p +a(x)|\nabla u_p|^{q-2} \nabla u_p \right) \cdot \nabla v \, d x \\
    \nonumber &\le& \|\nabla u_p\|_p^{\frac{p}{p'}} \, \|\nabla v\|_p + \|\nabla u_p\|_{L_{a}^{q}}^{\frac{q}{q'}} \, \|\nabla v\|_{L_{a}^{q}}\\
    \nonumber &\le&  \|\nabla u_p\|_{\theta_p}^{\frac{p}{p'}} \, \|\nabla v\|_p + \|\nabla u_p\|_{\theta_p}^{\frac{q}{q'}} \, \|\nabla v\|_{L_{a}^{q}}\\
    \label{rhsIp} &\le& \|\nabla u_p\|_{\theta_p}^{\frac{q}{q'}} \, \left(\|\nabla v\|_p + \|\nabla v\|_{L_{a}^{q}}\right)\,.
\end{eqnarray}
%%%%%%%%%%%%%%%%%%%%%%%%%%%%%%%%%%%%%%%%%%%%%%%%%%%%%
Letting $\lambda_p= \|\nabla u_p\|_{\theta_p}>1$, we estimate $\text{I}_p$ from below and get that
%%%%%%%%%%%%%%%%%%%%%%%%%%%%%
\begin{eqnarray}
\nonumber\text{I}_p &=&{\lambda_p}^{p} \int_\Omega \bigg|\frac{\nabla u_p}{\lambda_p}\bigg|^pd x + {\lambda_p}^{q} \int_{\Omega} a(x) \bigg|\frac{\nabla u_p}{\lambda_p}\bigg|^q\, d x\\
\nonumber&\ge &\lambda_p^p \int_\Omega \bigg(\bigg|\frac{\nabla u_p}{\lambda_p}\bigg|^p + a(x) \bigg|\frac{\nabla u_p}{\lambda_p}\bigg|^q\bigg)\, d x\\
\label{lhsIp}&=&\lambda_p^p.\
\end{eqnarray}
Combining inequalities \eqref{rhsIp} and \eqref{lhsIp} we have
\begin{equation*}
   {\lambda_p}^{p-\frac{q}{q'}}\le \|\nabla v\|_p + \|\nabla v\|_{L_{a}^{q}}
\end{equation*}
and by H\"{o}lder's inequality
\begin{equation*}
    {\lambda_p}^{p-\frac{q}{q'}} \le \left( |\Omega|^{1-\frac{1}{q}} +\|a\|_{\infty}^q \right) \|\nabla v\|_q.
\end{equation*}
Thus, for $p$ close to 1, we have
\begin{equation}\label{ubOrlicznorms}
    {\lambda_p} \le \left( (|\Omega|^{\frac{1}{q'}} \, +\|a\|_{\infty}^{\frac{1}{q}}) \, \|\nabla v\|_q \right)^{\frac{1}{{1-\frac{q}{q'}}}} +1=M.
\end{equation}
Hence, the norms $\|\nabla u_p\|_{\theta_p}$ are bounded by a constant independent of $p$, for $p$ close to 1. Using the same argument as in the proof of Proposition \ref{existencepq}, we have 
\begin{equation*}
   \|u_p\|_{\theta_p} \leq C\|\nabla u_p\|_{\theta_p} + C'\|\nabla v \|_{q} + \|v \|_{q} \le CM+C''=M_1\,.
\end{equation*}
and by \eqref{poincare} $M_1$ is also independent of $p$. 

\noindent Next, again by the embedding $L^{\theta_p}(\Omega)\hookrightarrow L^{p}(\Omega)$, we have that 
$$\|\nabla u\|_p \le  M,$$ for p close to 1. Thus, by H\"{o}lder's inequality we get
$$\int_{\Omega} |\nabla u_p| \, dx \le |\Omega|^{1-\frac{1}{p}} \|\nabla u_p\|_p \le  |\Omega|^{1-\frac{1}{p}} M \le (M+1)=M_2,$$ for $p$ close to 1. Moreover, since $u_p|_{\partial \Omega}=h$, we can estimate the equivalent $BV(\Omega)$ norm 
$$\|u_p\|= \int_{\Omega} |\nabla u_p| \, d x \, + \int_{\partial \Omega} |u_p| \, d \mathcal{H}^{N-1} \le M_2 + \int_{\partial \Omega} |h| \, d \mathcal{H}^{N-1}.$$ 

Hence $(u_p)$ is bounded in $BV(\Omega)$ and so by $BV(\Omega)$'s compactness property there exists $u\in BV(\Omega)$ such that, passing to a subsequence which we denote again as $(u_p)$, we get that
%%%%%%%%%%%%%%%%%%%%%%%%%%%%%%%%%%%%%%%%%%%%%%%%%%%%%%%%%%
\begin{equation*}
u_p\rightarrow u \, \text{ in } \, L^1(\Omega) \text{ and } \, Du_p\rightarrow Du\,  \text{ weak* as measures in }\, \Omega.
\end{equation*}
%%%%%%%%%%%%%%%%%%%%%%%%%%%%%%%%%%%%%%%%%%%%%%%%%%%%%%%%%%
Note that since for $1\le s <\frac{N}{N-1}$ the embedding $BV(\Omega) \hookrightarrow L^{s}(\Omega)$ is compact  we actually have that
\begin{equation*}
 u_p\rightarrow u \, \text{ in } L^{s}(\Omega)\,,\text{ for all }1\le s <\frac{N}{N-1}\,.
\end{equation*} 
Since, $W^{1,\theta_p}(\Omega)\hookrightarrow W_{a}^{1,q}(\Omega)$ we get that $(u_p)$ is also bounded in ${W^{1,q}_{a}}$
and hence by reflexivity we have that, after passing to a further subsequence, $$u_p \stackrel{w}{\rightarrow} u\,,\text{ in } W_{a}^{1,q}(\Omega)\,.$$ 
Note that this, by Remark \ref{Derivative} implies that $u\in W^{1,1}(\Omega)\cap W_{a}^{1,q}(\Omega)$ and by the continuity of the trace operator that $u|_{\partial \Omega}=h$.

Since by \eqref{ubOrlicznorms} we have that $\|\nabla u\|_{\theta_p} \le M$, for $p$ close to 1 and $1\leq r<p'$, we get by H\"{o}lder's inequality that 
$$\int_{\Omega} |\nabla u_p |^{(p-1)r}\,dx \le |\Omega|^{1-\frac{(p-1)r}{p}} M^{(p-1)r} $$ and hence 
%%%%%%%%%%%%%%%%%%%%%%%%%%%%%%%%%%%%%%%%%%%%%
\begin{equation}
    \label{infinity}
\||\nabla u_p|^{p-2} \nabla u_p \|_r \le (1+|\Omega|)^{\frac{1}{r}}\,.
\end{equation}
%%%%%%%%%%%%%%%%%%%%%%%%%%%%%%%%%%%%%%%%%%%%%
Note that for any fixed $r\geq 1$ by taking $p$ close enough to $1$ we get that $1\leq r<p'$. Hence by \eqref{infinity} the sequence  $(|\nabla u_p|^{p-2} \nabla u_p)$ is bounded in $L^r(\Omega)^N$ and thus it converges weakly to a $z_r\in L^r(\Omega)^N$. By a diagonal argument we may find a subsequence and a common vector field $z\in L^r(\Omega)^N$ such that
$$|\nabla u_p|^{p-2}\nabla u_p\stackrel{w}{\rightarrow} z\, \text{ in }L^r(\Omega)^N\,,$$
for all $1\le r < +\infty$. By \eqref{infinity} and using the fact that the norm is lower semicontinuous we get that
\[\|z\|_r\leq (1+|\Omega|)^{\frac{1}{r}}\]
and thus
\[\|z\|_\infty=\lim_{r\rightarrow \infty}\|z\|_r\leq 1\,.\]

The last two convergences of the Proposition follow as in the proof of \cite[Proposition 3.6]{MY} using Proposition \ref{prop:MS} and the fact that $L^q_a(\Omega)$ being uniformly convex has the Radon-Riesz property.
\end{proof}

We are now ready for our main result.

\begin{theorem}\label{Thm:Main} 
Assume that (H) holds. Then for each $h\in W^{1-\frac{1}{q},q}(\partial \Omega)$ there exists a weak solution to problem \eqref{limiting}.
\end{theorem}

\begin{proof} 
For any $\phi \in C_{0}^{\infty}(\Omega)$ by the weak formulation \eqref{weakformpq} we have that
\begin{equation*}
    \int_{\Omega} |\nabla u_p|^{p-2} \nabla u_p\cdot \nabla \phi\,dx  +\int_{\Omega} a(x)|\nabla u_p|^{q-2} \nabla u_p  \cdot \nabla \phi \, d x =0\,.
\end{equation*}
Letting $p\to 1$ and using the previous Proposition we get that
\begin{equation}
\label{weak}
    \int_{\Omega} z\cdot \nabla\phi\,dx +\int_{\Omega} a(x)|\nabla u|^{q-2} \nabla u\cdot \nabla \phi \, d x =0 \,,\text{ for all }\phi\in C^\infty_0(\Omega)\,.
\end{equation}

\noindent To complete the proof we need to show that 
\begin{equation*}
z \cdot \nabla u=|\nabla u| \, \text{ a.e. in } \, \Omega.
\end{equation*}

\noindent To this end let $\phi \in C_{0}^{\infty}(\Omega)$ with $\phi\geq 0$ and take $u_p \phi$ as a test function  in \eqref{weakformpq}. Then
%%%%%%%%%%%%%%%%%%%%%%%%%%%%%%%%%%%%%%%%%%%%%%%%%%%%%
\begin{eqnarray*}
\int_{\Omega} \phi|\nabla u_p |^{p} d x+ \int_{\Omega} u_p |\nabla u_p |^{p-2} \nabla u_p \cdot \nabla \phi \, d x+\int_{\Omega} a(x)\phi|\nabla u_p |^q \, d x+\\
+\int_{\Omega} a(x) u_p|\nabla u_p |^{q-2} \nabla u_p \cdot \nabla \phi \, d x=0.
\end{eqnarray*}
%%%%%%%%%%%%%%%%%%%%%%%%%%%%%%%%%%%%%%%%%%%%%%%%%%%%%
By Young's inequality we have that
\[\int_\Omega\phi|\nabla u_p|\,d x\leq\frac{1}{p}\int_\Omega\phi |u_p|^pd x+\frac{p-1}{p}\int_\Omega\phi\, d x\]
and hence from the previous equation we get that
%%%%%%%%%%%%%%%%%%%%%%%%%%%%%%%%%%%%%%%%%%%%%%%%%%%%%
\begin{eqnarray}
\nonumber p\int_{\Omega} \phi|\nabla u_p |\, d x+ \int_{\Omega} u_p |\nabla u_p |^{p-2} \nabla u_p \cdot \nabla \phi \, d x
 +\int_{\Omega} a(x)\phi|\nabla u_p |^q \, d x+\\
 \label{final}+\int_{\Omega} a(x) u_p|\nabla u_p |^{q-2} \nabla u_p \cdot \nabla \phi \, d x \leq (p-1)\int_\Omega\phi \,d x\,.   
\end{eqnarray}
%%%%%%%%%%%%%%%%%%%%%%%%%%%%%%%%%%%%%%%%%%%%%%%%%%%%%
By Proposition \ref{prop:behaviour} we get that
\[u_p|\nabla u_p |^{p-2} \nabla u_p\stackrel{w}{\rightarrow} u z\text{ in }  L^{s}(\Omega)^N\,,\text{ for all }1\le s <\frac{N}{N-1}\]
\noindent and hence
\begin{equation*}
    \int_{\Omega} u_p \, (|\nabla u_p |^{p-2} \nabla u_p \cdot \nabla \phi )\, d x \to \int_{\Omega} u \, (z \cdot \nabla \phi )\, d x.
\end{equation*}
\noindent Next we have that
\begin{equation*}
 \int_{\Omega} u_p  a(x) |\nabla u_p |^{q-2} \nabla u_p \cdot \nabla \phi \, d x - \int_{\Omega} u a(x) |\nabla u |^{q-2} \nabla u  \cdot \nabla \phi \, d x =
 \end{equation*}
 \begin{eqnarray}
 \nonumber&=&\int_{\Omega} (u_p-u) a(x) |\nabla u_p |^{q-2} \nabla u_p  \cdot \nabla \phi \, d x+ \\
 %\end{equation*}
 %\begin{equation*}
\label{second}&+&\int_{\Omega} u  a(x) \left(|\nabla u_p |^{q-2} \nabla u_p-|\nabla u |^{q-2} \nabla u \right)  \cdot \nabla \phi \, d x\,.
\end{eqnarray}
Again by Proposition \ref{prop:behaviour}
$$\nabla u_p\rightarrow\nabla u\text{ in }L^{q}_a(\Omega)^N\text{ and }|\nabla u_p |^{q-2} \nabla u_p\stackrel{w}{\rightarrow} |\nabla u |^{q-2} \nabla u\text{ in }L^{q'}_a(\Omega)^N $$
and thus the second summand in \eqref{second} converges to 0 as $p\rightarrow 1$.

\noindent For the first summand using H\"{o}lder's and Poincar\'{e}'s inequality we get that
\begin{eqnarray*}
    \int_{\Omega} (u_p-u)  a(x) |\nabla u_p |^{q-2} \nabla u_p \cdot \nabla \phi \, d x &\le& \|\nabla \phi\|_{\infty} \, \|u_p - u\|_{L^{a}_{q}} \, \|\nabla u_p\|_{L^{q}_{a}}^{\frac{q}{q'}}\\
    &\le& C\|\nabla u_p - \nabla u\|_{L^{a}_{q}} \, \|\nabla u_p\|_{L^{q}_{a}}^{\frac{q}{q'}} 
\end{eqnarray*}
and hence it also converges to 0.

\noindent Using the above and the lower semicontinuity  of the total variation we pass to the limit in \eqref{final} and get that
%%%%%%%%%%%%%%%%%%%%%%%%%%%%%%%%%%%%%%%%%%%%%%%%%%%%%
\begin{eqnarray}
\nonumber\int_{\Omega} \phi|\nabla u |\, d x+ \int_{\Omega} uz \cdot \nabla \phi \, d x
 +\int_{\Omega} a(x)\phi|\nabla u |^q \, d x+\\
\label{final1}  +\int_{\Omega} a(x) u|\nabla u |^{q-2} \nabla u \cdot \nabla \phi \, d x \leq 0\,.   
\end{eqnarray}
%%%%%%%%%%%%%%%%%%%%%%%%%%%%%%%%%%%%%%%%%%%%%%%%%%%%%
If $v\in W^{1,q}_0(\Omega)$ then by density and using that $W^{1,q}_0(\Omega)\hookrightarrow W^{1,q}_{a,0}(\Omega)$, we get from \eqref{weak} that  
%%%%%%%%%%%%%%%%%%%%%%%%%%%%%%%%%%%%%%%%%%%%%%%%%%%%%
\begin{equation}
\label{weakw1q}
    \int_{\Omega} z\cdot \nabla v\,dx +\int_{\Omega} a(x)|\nabla u|^{q-2} \nabla u\cdot \nabla v \, d x =0\,.
\end{equation}
%%%%%%%%%%%%%%%%%%%%%%%%%%%%%%%%%%%%%%%%%%%%%%%%%%%%%
Since $u\phi\in W_{0}^{1,1}(\Omega)\cap W^{1,q}_{a. 0}(\Omega)$ we have by Proposition \ref{prop:MS} and Remark \ref{remark} that there exists a sequence $(v_n)$ in $W^{1,q}_0(\Omega)$ such that 
\[v_n\to u\phi \text{ in } W^{1,1}(\Omega)\text{ and } \nabla v_n \to \nabla (u\phi)\text{ in } L^{q}_{a}(\Omega)^N\,.\]
Since each $v_n$ satisfies \eqref{weakw1q}, passing to the limit we get that
\[\int_{\Omega} z \cdot \nabla (u\phi) \, d x+ \int_{\Omega} a(x) |\nabla u|^{q-2} \nabla u \cdot \nabla (u\phi) \, d x = 0\,.\]

\noindent Combining this with \eqref{final1} we get that
%%%%%%%%%%%%%%%%%%%%%%%%%%%%%%%%%%%%%%%%%%%%%%%
\begin{equation*}
\int_{\Omega} \phi  \, |\nabla u | \, d x \le  \int_{\Omega} \phi  \, z \cdot \nabla u  \, d x, \, \text{ for all }\, \phi \in C^{\infty}_{0}(\Omega)
\end{equation*}
%%%%%%%%%%%%%%%%%%%%%%%%%%%%%%%%%%%%%%%%%%%%%%%%%
and hence
$$|\nabla u|\le z\cdot \nabla u\text{ a.e. in } \Omega\,.$$ 
Since on the other hand $\|z\|_{\infty}\le 1$ implies that $z \cdot \nabla u \le |\nabla u|$ we infer that
\begin{equation*}
    z \cdot \nabla u = |\nabla u| \, \text{ a.e. in }\, \Omega,
\end{equation*}
which concludes the proof. 
\end{proof}
%%%%%%%%%%%%%%%%%%%%%%%%%%%%%%%%%%%%%%%%%%%%%%%%%%%%%%%%%%%%%%%
\begin{remark}
\label{weakformulation}
(Weak formulation) If $u$ is a weak solution of the Dirichlet problem \eqref{limiting} we also have the following weak formulation
\begin{equation*}
\int_{\Omega} |\nabla u| \, d x - \int_{\Omega} z \cdot \nabla v \, d x+ \int_{\Omega} a(x) |\nabla u|^{q-2} \nabla u \cdot \nabla (u-v) \, d x = 0,\
\end{equation*}
for all $v\in W^{1,1}(\Omega)\cap W^{1,q}_{a}(\Omega)$, with $v|_{\partial \Omega}=h$.
\end{remark}
%%%%%%%%%%%%%%%%%%%%%%%%%%%%%%%%%%%%%%%%%%%%%%%%%%%%%%%%%%%%%%%%%
\begin{proposition} 
\label{uniqueness}
Under the assumptions (H), the solution of problem \eqref{limiting} is unique.
    
\end{proposition}

\begin{proof}
    Let $u_1,u_2\in W^{1,1}(\Omega)\cap W_a^{1,q}(\Omega)$ be two solutions of \eqref{limiting}. Hence there exist two vector fields $z_1,\,z_2\in L^\infty(\Omega)^N$ such that the conditions of Definition \ref{Notion} are satisfied. Testing with $u_2$ in the weak formulation for $u_1$ and vice versa we obtain
%%%%%%%%%%%%%%%%%%%%%%%%%%%%
\begin{equation}
\nonumber
\int_\Omega |\nabla u_1| \, d x - \int_\Omega z_1 \cdot \nabla u_2 \, d x + \int_\Omega a(x) |\nabla u_1|^{q-2} \nabla u_1 \cdot \nabla (u_1-u_2)\, d x = 0
\end{equation}
%%%%%%%%%%%%%%%%%%%%%%%%%%%%%%%%%%
and
%%%%%%%%%%%%%%%%%%%%%%%%%%
\begin{equation}
\nonumber
\int_\Omega |\nabla u_2| \, d x- \int_\Omega z_2 \cdot \nabla u_1 \, d x + \int_\Omega a(x) |\nabla u_2|^{q-2} \nabla u_2 \cdot \nabla (u_2-u_1)\, d x = 0\,.
\end{equation}
%%%%%%%%%%%%%%%%%%%%%%%%%%
Adding the above equations we get
%%%%%%%%%%%%%%%%%%%%%%%%%%%%%%%%%%%%%
\begin{eqnarray*}
\int_\Omega |\nabla u_1| \, d x&+&\int_\Omega |\nabla u_2| \, d x -\int_\Omega z_1 \cdot \nabla u_2 \, d x -\int_\Omega z_2 \cdot \nabla u_1 \, d x +\\
&+ &\int_\Omega a(x) \left(|\nabla u_1|^{q-2} \nabla u_1-|\nabla u_2|^{q-2} \nabla u_2 \right) \nabla (u_1-u_2)\, d x =0.\
\end{eqnarray*}
%%%%%%%%%%%%%%%%%%%%%%%%%%
Since $\|z_1\|_\infty \le 1$ and $\|z_2\|_\infty \le 1$ we have that 
\begin{equation*}
\int_\Omega z_1 \cdot \nabla u_2 \, d x \le \int_\Omega |\nabla u_2|\, d x \, \text{ and } \, 
\int_\Omega z_2 \cdot \nabla u_1 \, d x \le \int_\Omega |\nabla u_1| \, d x
\end{equation*}
%%%%%%%%%%%%%%%%%%%%%%%%%%%%%%%%%%%%%5
and hence
\begin{equation*}
\int_\Omega a(x) \left(|\nabla u_1|^{q-2} \nabla u_1-|\nabla u_2|^{q-2} \nabla u_2 \right) \cdot \nabla (u_1-u_2)\, d x \le 0\,.
\end{equation*}
But since the integrand is non-negative this implies that
\begin{equation*}
\int_\Omega a(x) \left(|\nabla u_1|^{q-2} \nabla u_1-|\nabla u_2|^{q-2} \nabla u_2 \right) \cdot \nabla (u_1-u_2)\, d x = 0\,.
\end{equation*}
Hence we conclude that
\[\nabla u_1=\nabla u_2\,,\text{ a.e. in }\Omega\]
and since $u_{1}|_{\partial \Omega}=u_{2}|_{\partial \Omega}$, Poincar\'e's inequality yields $u_1=u_2$.
\end{proof}
%%%%%%%%%%%%%%%%%%%%%%%%%%%%%%%%%%%%%%%%%%%%%%%%%%
%%%%%%%%%%%%%%%%%%%%%%%%%%%%%%%%%%%%%%%%%%%%%%%%%%
%%%%%%%%%%%%%%%%%%%%%%%%%%%%%%%%%%%%%%%%%%%%%%%%%%
We conclude this paper with a variational characterization of the solution of problem \eqref{limiting}. In particular we show that it is the unique minimizer of the minimization problem
%%%%%%%%%%%%%%%%%%%%%%%%%%%%%%%%%%%%%%%%%%
\begin{equation}
\nonumber
\min\left\{\mathcal{I}(u):\, u\in W^{1,1}(\Omega)\cap W^{1,q}_{a}(\Omega) \text{ with } u|_{\partial \Omega}=h 
\right\},
\end{equation}
where 
%%%%%%%%%%%%%%%%%%%%%%%%%%%%%%%%%%%%%%%%%%%%%%%%%%
\begin{equation}
\nonumber
\mathcal{I}(u) =\int_{\Omega}|\nabla u |\, d x+\frac{1}{q} \int_{\Omega} a(x) |\nabla u |^{q} d x\,.
\end{equation}
%%%%%%%%%%%%%%%%%%%%%%%%%%%%%%%%%%%%%%%%%%%%%%%%%
We have the following.
\begin{proposition} 
The function $u\in W^{1,1}(\Omega)\cap W^{1,q}_{a}(\Omega)$ is the unique weak solution of problem \eqref{limiting} if and only if it is the unique minimizer of the functional $\mathcal{I}$.
\end{proposition}
%%%%%%%%%%%%%%%%%%%%%%%%%%%%%%%%%%%%%%%%%%%%%%%%%%
\begin{proof}
The solution $u$ of problem \eqref{limiting}, by Theorem \ref{Thm:Main} and Proposition \ref{uniqueness}, exists and is unique. Since by strict convexity, the minimizer of $\mathcal{I}$ is also unique it is enough to show that $u$ is a minimizer of $\mathcal{I}$.

To this end let $v\in W^{1,1}(\Omega)\cap W^{1,q}_{a}(\Omega)$ with $v|_{\partial \Omega}=h$. By the weak formulation of Remark \ref{weakformulation} we have that
\begin{equation*}
\int_{\Omega} |\nabla u| \, d x - \int_{\Omega} z \cdot \nabla v \, d x+ \int_{\Omega} a(x) |\nabla u|^{q-2} \nabla u \cdot \nabla (u-v) \, d x = 0\,.
\end{equation*}
Using the fact that $\|z \|_{\infty} \le 1$ and Young's inequality we get that
\begin{eqnarray*}
\int_{\Omega} |\nabla u| \, d x +\int_{\Omega} a(x) |\nabla u|^q\, d x= \int_{\Omega} z \cdot \nabla v \, d x+ \int_{\Omega} a(x) |\nabla u|^{q-2} \nabla u \cdot \nabla v \, d x\\
\leq\int_{\Omega} |\nabla v| \, d x+\frac{1}{q'} \int_{\Omega} a(x) |\nabla u |^{q} d x + \frac{1}{q} \int_{\Omega} a(x) |\nabla v |^{q} d x\,.
\end{eqnarray*}
But this implies that
\begin{equation*}
\int_{\Omega}|\nabla u |\, d x+\frac{1}{q} \int_{\Omega} a(x) |\nabla u |^{q}\, d x\leq \int_{\Omega}|\nabla v |\, d x+\frac{1}{q} \int_{\Omega} a(x) |\nabla v |^{q}\, d x
\end{equation*}
i.e.
\begin{equation*}
\mathcal{I}(u) \le \mathcal{I}(v)\,.
\end{equation*}

\end{proof}

\noindent {\bf Acknowledgments.} The first author has been supported by the Special Account for Research Funding of the National Technical University of Athens.
%%%%%%%%%%%%%%%%%%%%%%%%%%%%%%%%%%%%%%%%%%%%%%%%%%%%%%%%%%%%%%%%%%%%%%%%
%%%%%%%%%%%%%%%%%%%%%%%%%%%%%%%%%%%%%%%%%%%%%%%%%%%%%%%%%%%%%%%%%%%%%%%%
%%%%%%%%%%%%%%%%%%%%%%%%%%%%%%%%%%%%%%%%%%%%%%%%%%%%%%%%%%%%%%%%%%%%%%%%

\end{document}